\documentclass[11pt,letterpaper]{amsart}

\newcommand{\la}{\langle}
\newcommand{\ra}{\rangle}
\renewcommand{\Re}{\operatorname{Re}}
\renewcommand{\Im}{\operatorname{Im}}

\newcommand{\supp}{\operatorname{supp}}

\usepackage{amssymb}
\usepackage{amsthm}
\usepackage{amsxtra}
\usepackage{graphicx}
\usepackage{mathabx}

\newtheorem{theorem}{Theorem}

\newtheorem{proposition}[theorem]{Proposition}
\newtheorem{lemma}[theorem]{Lemma}
\newtheorem{corollary}[theorem]{Corollary}

\theoremstyle{remark}
\newtheorem{remark}[theorem]{Remark}

\numberwithin{equation}{section}

\numberwithin{theorem}{section}

\numberwithin{table}{section}

\numberwithin{figure}{section}

\ifx\pdfoutput\undefined
  \DeclareGraphicsExtensions{.pstex, .eps}
\else
  \ifx\pdfoutput\relax
    \DeclareGraphicsExtensions{.pstex, .eps}
  \else
    \ifnum\pdfoutput>0
      \DeclareGraphicsExtensions{.pdf}
    \else
      \DeclareGraphicsExtensions{.pstex, .eps}
    \fi
  \fi
\fi

\title[GWP below Energy Norm]{Global Well-posedness of NLS with a Rough Potential below the Energy Norm}

\date{\today}
\author{Younghun Hong}
\address{Brown University}
\begin{document}

\maketitle

\begin{abstract}
We show that a 3d cubic defocusing nonlinear Schr\"odinger equation with a potential is globally well-posed in $H^s$ for $s>\frac{5}{6}$, provided that a potential is contained in $B\cap L^\infty$ and its negative part has small global Kato norm, where
$$B=\Big\{V: \sum_{k=-\infty}^\infty2^{k/2}\|V\|_{L^2(2^k\leq|x|<2^{k+1})}<\infty\Big\}.$$
The proof is based on the approach of Colliander-Keel-Staffilani-Takaoka-Tao \cite{CKSTT1}, called the \textit{$I$-method}, but in order to deal with a rough potential, we modify harmonic analysis tools by spectral theory.
\end{abstract}

\section{Introduction}

In this paper, we consider a 3d cubic defocusing nonlinear Schr\"odinger equation
\begin{equation}\tag{$\textup{NLS}_V$}
iu_t+\Delta u-Vu-|u|^2u=0;\ u(0)=u_0\in H^s,
\end{equation}
where $u=u(t,x)$ is complex-valued, $(t,x)\in\mathbb{R}\times\mathbb{R}^3$ and $V=V(x)$ is a real-valued potential. It is known that for a large class of short-range potentials, $\textup{NLS}_V$ is locally(-in-time) well-posed in $H^s$ for $s\in(\frac{1}{2},1]$ (see \cite{H2}). Such solutions satisfy the mass conservation law 
$$M[u(t)]=\int_{\mathbb{R}^3}|u(t,x)|^2dx=M[u_0],$$
and $H^1$ solutions satisfy the energy conservation law
$$E[u(t)]=\frac{1}{2}\int_{\mathbb{R}^3}|\nabla u(t,x)|^2+V(x)|u(t,x)|^2 dx+\frac{1}{4}\int_{\mathbb{R}^3}|u(t,x)|^4 dx=E[u_0]$$
during their existence time. The uniform bound from these conservation laws then yields global(-in-time) well-posedness in $H^1$.

Our goal is to lower the required regularity for global well-posedness to that for local well-posedness $(s>\frac{1}{2})$. At the same time, we also aim to include as large a potential class as possible. Indeed, the largest potential class we may hope for is $\mathcal{K}_0$, that is, the norm closure of bounded and compactly supported functions with respect to the global Kato norm
$$\|V\|_{\mathcal{K}}:=\sup_{x\in\mathbb{R}^3} \int_{\mathbb{R}^3}\frac{|V(y)|}{|x-y|}dy,$$
since it is the largest class for which Strichartz estimates are known \cite{BG, GVV}.

In the homogeneous case $V=0$, Colliander-Keel-Staffilani-Takaoka-Tao established global well-posedness in $H^s$ for $s>\frac{5}{6}$, using the \textit{$I$-method} \cite{CKSTT1}, which improves Bourgain's \textit{Fourier truncation method} \cite{Bo}: they showed that the energy of the smoothed solution $Iu$, controlling the $H^s$ norm, grows at most polynomially in time. Later, in \cite{CKSTT2}, utilizing the \textit{interaction Morawetz estimate}, the same authors pushed down the required regularity to $s>\frac{4}{5}$. Currently, this is the best known result for the homogeneous equation. However, as far as the author knows, nothing has been proved for inhomogeneous ones.

In this paper, we first show global well-posedness below the energy norm in the presence of a potential. We define the potential class $B$ by 
$$B:=\Big\{V: \sum_{k=-\infty}^\infty2^{k/2}\|V\|_{L^2(2^k\leq|x|<2^{k+1})}<\infty\Big\},$$
and denote the negative part of the potential $V$ by $V_-(x):=\min(V(x),0)$. Then, we prove:
\begin{theorem}[Global well-posedness]\label{thm:MainTheorem}
If $V\in B\cap L^\infty$ and $\|V_-\|_{\mathcal{K}}<4\pi$, then $\textup{NLS}_V$ is globally well-posed in $H^s$ for $s>\frac{5}{6}$.
\end{theorem}

\begin{remark}[Potential class]
$(i)$ The assumptions in Theorem 1.1 allow us to use the two main ingredients of the proof: high frequency approximation lemmas (Lemma 2.1 and 2.2) and Beceanu's structure formula for the wave operator \cite{B}. The smallness of $V_-$ guarantees coercivity of the energy $E[u]$ as well as the absence of zero resonances (Lemma A.1).\\
$(ii)$ We do not assume any differentiability of a potential. For example, a potential with the bounds $0\leq V(x)\lesssim\la x\ra^{-2-}$ satisfies the hypotheses of the theorem.
\end{remark}

\begin{remark}[Morawetz type inequalties]
The current technology of Morawetz type inequalities for $\textup{NLS}_V$ always requires smallness of the confining part of a potential, $\max(\frac{x}{|x|}\cdot\nabla V,0)$ \cite{FV, CCL}. Thus they cannot be applied to $\textup{NLS}_V$ with potentials having large wells or highly oscillatory potentials, for example. In order to include a larger class of potentials, we avoid using the interaction Morawetz estimate as in \cite{CKSTT2}.
\end{remark}

We prove Theorem 1.1 using the $I$-method, but we also make several adjustments to deal with rough potentials. When a potential is smooth and rapidly decaying, it is conventional to treat the linear term $Vu$ as a nonlinear term, and use the Duhamel formula 
\begin{equation}
u(t)=e^{it\Delta}u_0-i\int_0^te^{i(t-s)\Delta}(Vu+|u|^2u)(s)ds
\end{equation}
(see \cite[Section 4]{C}). However, if a potential is rough, the Duhamel formula $(1.1)$ does not make sense, since $Vu$ has infinite Sobolev norm even for smooth $u$. For this reason, considering a potential as a part of the the linear operator, we use the Duhamel formula
\begin{equation}
u(t)=e^{-itH}u_0-i\int_0^te^{-i(t-s)H}(|u|^2u)(s) ds.
\end{equation}
Now we observe that the $I$-operator of \cite{CKSTT1} is not applicable to $(1.2)$, because a Fourier multiplier $I$ does not commute with the linear propagator $e^{-itH}$. To solve this problem, we replace the $I$-operator by the $\mathcal{I}$-operator, namely a smoothed high \textit{spectrum} truncation defined as a spectral multiplier. We then run the program of Colliander-Keel-Staffilani-Takaoka-Tao \cite{CKSTT1} together with the following two observations. First, in high spectrum, spectral multipliers can be approximated by Fourier multipliers (see Section 2). Second, the wave operator is a convenient tool to derive linear and bilinear estimates associated with $H$ (see Section 3).

\subsection{Organization of the paper}
In Section 2 and 3, we present the key items of this paper: high frequency approximation lemmas and the wave operator. In Section 4 and 5, we prove the almost conservation law and the main theorem.

\subsection{Notations}
We fix a standard dyadic partition of unity function $\chi\in C_c^\infty(\mathbb{R})$ such that $\chi$ is supported in $[\frac{1}{2},2]$ and $\sum_{N\in 2^{\mathbb{Z}}}\chi(\tfrac{\cdot}{N})\equiv 1$ on $\mathbb{R}^+$. We define standard Littlewood-Paley projections by $\widehat{P_Nf}(\xi)=\chi_N(|\xi|)\hat{f}(\xi)$. When $V\in B$ and $\|V_-\|_{\mathcal{K}}<4\pi$, using functional calculus, we define \textit{perturbed} Littlewood-Paley projections by $\mathcal{P}_N=\chi_N(\sqrt{H})$.

\subsection{Acknowledgement} The author would like to thank his advisor, Justin Holmer, for his help and encouragement.

\section{High Frequency Approximation}

\subsection{Heuristic argument}
We begin by a heuristic argument involving the distorted Fourier transform. Let $e_V(x,\xi)$ be the solutions to the Lippmann-Schwinger equation
\begin{equation}
e_V(x,\xi)=e^{ix\cdot\xi}-\int_{\mathbb{R}^3}\frac{e^{i|\xi||x-y|}}{4\pi|x-y|} V(y) e_V(y,\xi) dy\underset{\textup{formally}}\Longleftrightarrow He_V(x,\xi)=|\xi|^2e_V(x,\xi),
\end{equation}
assuming that the potential $V$ is good enough to guarantee solvability of $(2.1)$ for all $\xi\in\mathbb{R}^3$. The collection $\{e_V(x,\xi)\}_{\xi\in\mathbb{R}^3}$ is called the \textit{distorted Fourier basis}, and the \textit{distorted Fourier transform} is defined by
$$(\mathcal{F}_Vf)(\xi):=\int_{\mathbb{R}^3} f(x)\overline{e_V(x,\xi)}dx.$$
The distorted Fourier transform is useful to analyze the linear propagator $e^{-itH}$, see \cite{I, A} for example. However, it is not a convenient tool in nonlinear analysis, mainly because the simple convolution property $\mathcal{F}_V(uv)=\mathcal{F}_Vu*\mathcal{F}_Vv$
is no longer available. Indeed, it is equivalent to the group structure of the Fourier basis  $e_V(x,\xi)e_V(y,\xi)=e_V(x+y,\xi)$ for all $x,y,\xi\in\mathbb{R}^3,$
but this is not expected to be true for general Schr\"odinger operators. 

We circumvent lack of the convolution property by considering distorted Fourier multipliers $\mathcal{F}_V^{-1}(\varphi(|\xi|)\mathcal{F}_Vf)$ with a symbol $\varphi\in C_c^\infty((0,+\infty))$. Observe that if $|\xi|$ is large, because of high oscillation, the standard Fourier basis $e^{ix\cdot\xi}$ \textit{almost} solves the eigenvalue equation $(2.1)$ in distribution sense:
$$\la(-\Delta+V) e^{ix\cdot\xi}, \psi\ra_{L^2}=\la|\xi|^2e^{ix\cdot\xi}, \psi\ra_{L^2}+\la Ve^{ix\cdot\xi}, \psi\ra_{L^2}\approx\la|\xi|^2e^{ix\cdot\xi}, \psi\ra_{L^2}\textup{ for all }\psi\in C_c^\infty.$$
Therefore, at least in high frequency, distorted Fourier multipliers can be approximated by standard ones. For example, we may guess that
$$(\varphi_N(|\xi|)\hat{f})^\vee-\mathcal{F}_V^{-1}(\varphi_N(|\xi|)\mathcal{F}_Vf)\to 0\textup{ as }N\to\infty$$
in some sense, where $\varphi\in C_c^\infty(\mathbb{R})$, $\supp\varphi\subset[1,2]$ and $\varphi_N:=\varphi(\frac{\cdot}{N})$, and that the convergence rate would be faster when a potential is more regular. If these are true, lack of the convolution property can be overcome in many situations.

\subsection{Formulation of the heuristic argument} We will make the above heuristics rigorous. Suppose that $V\in B\cap L^\frac{3}{2-\alpha}$ with $0<\alpha\leq2$ and $\|V_-\|_{\mathcal{K}}<4\pi$. For $\varphi_N$ as above, we consider the spectral multiplier $\varphi_N(\sqrt{H})$, defined by functional calculus. Note that if both the distorted Fourier transform and its inverse transform are well-defined, then a spectral multiplier is simply a distorted Fourier multiplier  $\varphi_N(\sqrt{H})f=\mathcal{F}_V^{-1}(\varphi_N(|\xi|)\mathcal{F}_Vf)$. From now, forgetting about the distorted Fourier transform, we only consider spectral multipliers, and we then aim to show that 
$$\varphi_N(\sqrt{-\Delta})-\varphi_N(\sqrt{H})$$
is \textit{small} for large $N$. This approach has two advantages: First, spectral multipliers are defined for any self-adjoint operator by functional calculus, so we do not need to worry about solvability of the Lippmann-Schwinger equation. Second, it allows us to avoid using the distorted Fourier basis. Indeed, this basis is very difficult to deal with, since $e(x,\xi)$'s are given only implicitly.

For notational convenience, we denote the difference of a Fourier multiplier and a spectral multiplier by
$$\mathfrak{D}_{\varphi,N}:=\varphi_N(\sqrt{-\Delta})-\varphi_N(\sqrt{H}).$$
The following two propositions explain the relation between local regularity of potentials and convergence rate:

\begin{lemma}[High frequency approximation (I)]\label{Approximation1} Suppose that $V\in\mathcal{K}_0\cap L^{\frac{3}{2-\alpha}}$ with $\alpha\in(0, 2]$ and $\|V_-\|_{\mathcal{K}}<4\pi$. Let $\varphi\in C_c^\infty$ with $\supp\varphi\subset[1,2]$ and $\varphi_N=\varphi(\frac{\cdot}{N})$. Then, for $N\gg1$,
\begin{equation}
\|\mathfrak{D}_{\varphi,N}\|_{L^p\to L^p}\lesssim N^{-\alpha},\ 1\leq p\leq\infty.
\end{equation}
\end{lemma}

\begin{proof}
By duality, it suffices to show $(2.2)$ for $p=1$. For $z\notin\textup{Spec}(H)$, we define the resolvent by
$$R_V^\pm(\lambda):=\underset{{\epsilon\to0+}}{\textup{s-}\lim} (H-\lambda\pm i\epsilon)^{-1}.$$
By the identity
$$R_0^+(\lambda)-R_V^+(\lambda)=R_0^+(\lambda)(I-(I+VR_0^+(\lambda))^{-1})=R_0^+(\lambda)(I+VR_0^+(\lambda))^{-1}VR_0^+(\lambda)$$
and the Stone's formula, $\mathfrak{D}_{\varphi,N}=\varphi_N(\sqrt{-\Delta})-\varphi_N(\sqrt{H})$ is represented by
\begin{equation}
\begin{aligned}
\mathfrak{D}_{\varphi,N}&=\frac{1}{\pi}\int_0^\infty\varphi_N(\sqrt{\lambda})\Im R_0^+(\lambda)d\lambda-\frac{1}{\pi}\int_0^\infty \varphi_N(\sqrt{\lambda})\Im R_V^+(\lambda)d\lambda\\
&=\frac{1}{\pi}\int_0^\infty \varphi_N(\sqrt{\lambda})\Im[R_0^+(\lambda)(I+VR_0^+(\lambda))^{-1}VR_0^+(\lambda)]d\lambda.
\end{aligned}
\end{equation}
Then, arguing as in \cite[Section 4]{H1}, one can find $K(\tilde{x},\tilde{y})\in L_{\tilde{y}}^\infty L_{\tilde{x}}^1$ such that
$$|\mathfrak{D}_{\varphi,N}(x,y)|\leq\iint_{\mathbb{R}^6}\frac{N^2K(\tilde{x},\tilde{y})|V(\tilde{y})|}{|x-\tilde{x}|\la N(x-\tilde{x})\ra^3|\tilde{y}-y|\la N(\tilde{y}-y)\ra^3}d\tilde{x}d\tilde{y}$$
for sufficiently large $N$. Applying the H\"older inequality and the Minkowski inequality, we prove that 
\begin{align*}
\|\mathfrak{D}_{\varphi,N}f\|_{L^1}&\leq\Big\|\iiint_{\mathbb{R}^9}\frac{N^2K(\tilde{x},\tilde{y})|V(\tilde{y})||f(y)|}{|x-\tilde{x}|\la N(x-\tilde{x})\ra^3|\tilde{y}-y|\la N(\tilde{y}-y)\ra^3}d\tilde{x}d\tilde{y}dy\Big\|_{L_x^1}\\
&\leq\iiint_{\mathbb{R}^9}\Big\|\frac{N^2}{|x-\tilde{x}|\la N(x-\tilde{x})\ra^3}\Big\|_{L_x^1}\frac{K(\tilde{x},\tilde{y})|V(\tilde{y})||f(y)|}{|\tilde{y}-y|\la N(\tilde{y}-y)\ra^3}d\tilde{x}d\tilde{y}dy\\
&\lesssim\iint_{\mathbb{R}^6}\|K(\tilde{x},\tilde{y})\|_{L_{\tilde{x}}^1}\frac{|V(\tilde{y})||f(y)|}{|\tilde{y}-y|\la N(\tilde{y}-y)\ra^3}d\tilde{y}dy\\
&\leq\|V\|_{L^{\frac{3}{2-\alpha}}}\Big\|\int_{\mathbb{R}^3}\frac{|f(y)|}{|\tilde{y}-y|\la N(\tilde{y}-y)\ra^3}dy\Big\|_{L_{\tilde{y}}^{\frac{3}{1+\alpha}}}\\
&\leq\int_{\mathbb{R}^3}\Big\|\frac{1}{|\tilde{y}-y|\la N(\tilde{y}-y)\ra^3}\Big\|_{L_{\tilde{y}}^{\frac{3}{1+\alpha}}}|f(y)|dy\lesssim N^{-\alpha}\|f\|_{L^1}.
\end{align*}
\end{proof}

\begin{lemma}[High frequency approximation (II)]\label{Approximation2} Suppose that $V\in\mathcal{K}_0\cap L^{\frac{3}{2-\alpha},\infty}$ with $\alpha\in(0,2]$ and $\|V_-\|_{\mathcal{K}}<4\pi$. Let $\varphi\in C_c^\infty$ with $\supp\varphi\subset[1,2]$ and $\varphi_N=\varphi(\frac{\cdot}{N})$. Then, for $N\gg1$ and $\beta\in(0,1]$,
\begin{equation}
\||\nabla|^\beta \mathfrak{D}_{\varphi,N}\|_{L^p\to L^p}\lesssim N^{-\alpha+\beta},\ 1< p<\infty.
\end{equation}
\end{lemma}

\begin{proof}
Consider the case $\beta=1$. Since $\||\nabla|f\|_{L^p}\sim \|\nabla f\|_{L^p}$ for $1<p<\infty$, by interpolation, it suffices to show that 
$$\|\nabla \mathfrak{D}_{\varphi,N}\|_{L^p\to L^p}\lesssim N^{-\alpha}\textup{ for }p=1,\infty.$$
We take $\nabla$ to $(2.3)$. Then, since the kernel of $\nabla R_0^+(\lambda)$ is given by
$$(\nabla R_0^+(\lambda))(x,y)=\nabla_x\Big(\frac{e^{i\sqrt{\lambda}|x-y|}}{4\pi|x-y|}\Big)=\frac{i\sqrt{\lambda}(x-y)e^{i\sqrt{\lambda}|x-y|}}{4\pi|x-y|^2}-\frac{(x-y)e^{i\sqrt{\lambda}|x-y|}}{4\pi|x-y|^3},$$
by the above argument, one can find $K_1(\tilde{x},\tilde{y}),K_2(\tilde{x},\tilde{y})\in L_{\tilde{y}}^\infty L_{\tilde{x}}^1$ such that 
\begin{align*}
|\nabla_x\mathfrak{D}_{\varphi,N}(x,y)|&\leq\iint_{\mathbb{R}^6}\frac{N^3K_1(\tilde{x},\tilde{y})|V(\tilde{y})|}{|x-\tilde{x}|\la N(x-\tilde{x})\ra^3|\tilde{y}-y|\la N(\tilde{y}-y)\ra^3}d\tilde{x}d\tilde{y}\\
&+\iint_{\mathbb{R}^6}\frac{N^2K_2(\tilde{x},\tilde{y})|V(\tilde{y})|}{|x-\tilde{x}|^2\la N(x-\tilde{x})\ra^3|\tilde{y}-y|\la N(\tilde{y}-y)\ra^3}d\tilde{x}d\tilde{y}.
\end{align*}
Thus, it follows that
$$\|\nabla \mathfrak{D}_{\varphi,N}\|_{L^1\to L^1}\lesssim N^{-\alpha+1}.$$
For $p=\infty$, by duality, it suffices to show
$$(\nabla \mathfrak{D}_{\varphi,N})^*=\frac{1}{\pi}\int_0^\infty \varphi_N(\sqrt{\lambda})\Im[R_0^+(\lambda)(I+VR_0^+(\lambda))^{-1}VR_0^+(\lambda)\nabla]d\lambda$$
is bounded on $L^1$. But, by integration by parts, we have
$$(R_0^+(\lambda)\nabla)(x,y)=-\nabla_y\Big(\frac{e^{i\sqrt{\lambda}|x-y|}}{4\pi|x-y|}\Big)=\nabla_x\Big(\frac{e^{i\sqrt{\lambda}|x-y|}}{4\pi|x-y|}\Big).$$
Hence, repeating the same procedure, we prove that 
$$\|\nabla \mathfrak{D}_{\varphi,N}\|_{L^\infty\to L^\infty}=\|(\nabla \mathfrak{D}_{\varphi,N})^*\|_{L^1\to L^1}\lesssim N^{-\alpha+1}.$$
Finally, by interpolation between $(2.2)$ and $(2.4)$ with $\beta=1$, we complete the proof.
\end{proof}

Similarly, in high frequency, $|\nabla|^{-1}$ and $H^{1/2}$ are cancelled out up to small error:

\begin{corollary} If $V\in\mathcal{K}_0\cap L^{\frac{3}{2-\alpha},\infty}$ with $\alpha\in(0,2]$ and $\|V_-\|_{\mathcal{K}}<4\pi$, then for $N\gg1$, 
$$\|P_N|\nabla|^{-1}H^{1/2}-\mathcal{P}_N\|_{L^p\to L^p}\lesssim N^{-\alpha},\ \tfrac{3}{2}<p<\infty.$$
\end{corollary}

\begin{proof}
First, we write
\begin{align*}
&\la(P_N|\nabla|^{-1}H^{1/2}-\mathcal{P}_N)f, g\ra_{L^2}=\la f, (H^{1/2}|\nabla|^{-1}P_N-\mathcal{P}_N)g\ra_{L^2}\\
&=\la f, H^{1/2}(|\nabla|^{-1}P_N-H^{-\frac{1}{2}}\mathcal{P}_N)g\ra_{L^2}\leq \|f\|_{L^p}\|H^{1/2}(|\nabla|^{-1}P_N-H^{-\frac{1}{2}}\mathcal{P}_N)g\|_{L^{p'}}.
\end{align*}
Then, by the norm equivalence (Lemma 3.5 below) with $1<p'<\frac{3}{2}$, we get 
$$\|H^{1/2}(|\nabla|^{-1}P_N-H^{-\frac{1}{2}}\mathcal{P}_N)g\|_{L^{p'}}\lesssim\||\nabla|(|\nabla|^{-1}P_N-H^{-\frac{1}{2}}\mathcal{P}_N)g\|_{L^{p'}}.$$
Observe that $|\nabla|^{-1}P_N$ and $H^{1/2}\mathcal{P}_N$ have the same symbol $\lambda^{-1}\chi_N$. It thus follows from Lemma \ref{Approximation1} that 
$$\|H^{1/2}(|\nabla|^{-1}P_N-H^{-\frac{1}{2}}\mathcal{P}_N)g\|_{L^{p'}}\lesssim N^{-1}N^{-\alpha+1}\|g\|_{L^{p'}}=N^{-\alpha}\|g\|_{L^{p'}}.$$
By duality, we obtain the corollary. 
\end{proof}

\section{Wave Operators and their Applications}
We define the (forward-in-time) \textit{wave operator} by
$$W:=\underset{t\to+\infty}{s\mbox{-}\lim}e^{itH}e^{-it(-\Delta)}.$$
The wave operator is a useful tool because of its intertwining property: for any Borel function $m:\mathbb{R}\to\mathbb{C}$,
\begin{equation}
m(H)P_c=W m(-\Delta) W^*,
\end{equation}
where $P_c$ is the spectral projection to the continuous spectrum and $W^*$ is the dual of $W$.

In \cite{B}, Beceanu obtained the structure formula for the wave operator. Let $O(3)=\{S\in\mathcal{B}(\mathbb{R}^3,\mathbb{R}^3): S^*S=I\}$ be the group of orthogonal linear transformations on $\mathbb{R}^3$. We say that $\psi\in L^\infty$ is a \textit{zero resonance} if $(I+(-\Delta)^{-1}V)\psi=0$.

\begin{theorem}[Structure of the wave operator \cite{B}]\label{thm:Structure} If $V\in B$ and that $H$ does not have a zero resonance, then there exists a measure $g_{s,y}(x)$ such that
$$(Wf)(x)=f(x)+\int_{\mathbb{R}^3}\Big(\int_{O(3)} f(sx+y) d g_{s,y}(x) \Big)dy,$$
and $\int_{\mathbb{R}^3}\int_{O(3)} d\|g_{s,y}(x)\|_{L_x^\infty}dy<\infty$. Thus, $W$ is bounded on $L^p$ for $1\leq p\leq \infty.$
\end{theorem}

\begin{remark}
$(i)$ In \cite{B}, the definition of a zero resonance is slightly different, but it is essentially equivalent to that in Theorem 3.1 (see $(2.4)$ of \cite{B}).\\
$(ii)$ If $V\in B$ and $\|V_-\|_{\mathcal{K}}<4\pi$, $H$ has no zero resonance (so $P_c=0$) (Lemma A.1).
\end{remark}

By Theorem \ref{thm:Structure} and the intertwining property, one can easily derive the following estimates from the homogeneous  analogues:

\begin{lemma}[Spectral multiplier theorem] Suppose that $V\in B$ and $\|V_-\|_{\mathcal{K}}<4\pi$. Let $m:(0,+\infty)\to\mathbb{C}$ be a symbol of order zero, i.e., $|\partial_\lambda^km(\lambda)|\lesssim_k\lambda^{-k}$ for $\lambda>0$ and $0\leq k\leq 5$. Then,
$$\|m(\sqrt{H})\|_{L^p\to L^p}<\infty,\ 1<p<\infty.$$
\end{lemma}

\begin{proof}
By the intertwining property, the boundedness of $m(\sqrt{H})=Wm(-\Delta)W^*$ follows from the boundedness of the wave operator and the classical Fouirer multiplier theorem.
\end{proof}

\begin{remark}
The spectral multiplier theorem holds for a larger class of potentials \cite{H1}.  
\end{remark}

\begin{lemma}[Norm equivalence]\label{lem:NormEqv} If $V\in B$ and $\|V_-\|_{\mathcal{K}}<4\pi$, then for $s\in[0,2]$,
$$\|H^{\frac{s}{2}}f\|_{L^r}\sim\||\nabla|^sf\|_{L^r},\ 1<r<\tfrac{3}{s}.$$
\end{lemma}

\begin{proof}
By the spectral multiplier theorem, imaginary power operators $H^{i\alpha}$, with $\alpha\in\mathbb{R}$, are bounded on $L^p$. Applying the Stein's complex interpolation to the bounded analytic families of operators $H^z(-\Delta)^{-z}$ and $(-\Delta)^zH^{-z}$ on $\{z\in\mathbb{C}: 0\leq\Re z\leq1\}$, we derive the norm equivalence. See details in \cite{H1}.
\end{proof}

\begin{lemma}[Sobolev inequality] If $V\in B$ and $\|V_-\|_{\mathcal{K}}<4\pi$, then
$$\|H^{-\frac{s}{2}}f\|_{L^q}\lesssim\|f\|_{L^p},$$
where $1<p<q<\infty$, $0\leq s\leq 2$ and $\frac{1}{q}=\frac{1}{p}-\frac{s}{3}$.
\end{lemma}

\begin{proof}
The Sobolev inequality follows from the standard Sobolev inequality and the norm equivalence.
\end{proof}

We define the standard (perturbed, resp) $X_{s,b}$-norm defined by 
$$\|u\|_{X_{s,b}}:=\|\la\nabla\ra^s\la \tau+\Delta\ra^b(\mathcal{F}_tu)\|_{L_{\tau,x}^2}\ \Big(\|u\|_{\mathcal{X}_{s,b}}:=\|\la H^{1/2}\ra^s\la \tau-H\ra^b(\mathcal{F}_tu)\|_{L_{\tau,x}^2}\textup{, resp}\Big),$$
where $\mathcal{F}_tu$ is the temporal Fourier transform of $u$ and $\la H^{1/2}\ra^s\la \tau-H\ra^b$ is given by functional calculus. We also define 
$$\|u\|_{X_{s,b}^\delta}:=\inf_{v=u\textup{ on }[0,\delta]}\|v\|_{X_{s,b}}\ \Big(\|u\|_{\mathcal{X}_{s,b}^\delta}:=\inf_{v=u\textup{ on }[0,\delta]}\|v\|_{\mathcal{X}_{s,b}}\Big).$$

\begin{lemma}[Strichartz estimates] If $V\in B$ and $\|V_-\|_{\mathcal{K}}<4\pi$, then
$$\|F\|_{L_t^q L_x^r(\mathbb{R}\times\mathbb{R}^3)}\lesssim \|F\|_{\mathcal{X}_{0,1/2+}},$$
where $2\leq q,r\leq\infty$ and $\frac{2}{q}+\frac{3}{r}=\frac{3}{2}$.
\end{lemma}

\begin{proof}
By the Fubini theorem, the temporal Fourier transform $\mathcal{F}_t$ commutes with the wave operator (see the structure formula in Theorem 3.1). We thus have
\begin{equation}
\|F\|_{\mathcal{X}_{s,b}}=\|W\la\nabla\ra^s\la\tau+\Delta\ra^b W^* ((\mathcal{F}_t F)(\tau))\|_{L_{\tau,x}^2}=\|W^*F\|_{X_{s,b}}.
\end{equation}
Hence, Lemma 3.7 follows from boundedness of the wave operator and $(2.5)$ of \cite{CKSTT1}:
$$\|F\|_{L_t^q L_x^r(\mathbb{R}\times\mathbb{R}^3)}=\|WW^*F\|_{L_t^q L_x^r(\mathbb{R}\times\mathbb{R}^3)}\lesssim\|W^*F\|_{L_t^q L_x^r(\mathbb{R}\times\mathbb{R}^3)}\lesssim\|W^*F\|_{X_{0,1/2+}}=\|F\|_{\mathcal{X}_{0,1/2+}}.$$
\end{proof}

\begin{lemma}[Some estimates involving the $\mathcal{X}_{s,b}$-norm] If $V\in B$ and $\|V_-\|_{\mathcal{K}}<4\pi$, then
\begin{align}
\|e^{-itH}f\|_{\mathcal{X}_{1,1/2+}^\delta}&\lesssim\|f\|_{H^1},\\
\Big\|\int_0^te^{-i(t-s)H}F(s)ds\Big\|_{\mathcal{X}_{1,1/2+}^\delta}&\lesssim\|F\|_{\mathcal{X}_{1,-1/2+}^\delta},\\
\|F\|_{\mathcal{X}_{1,-b}^\delta}&\lesssim\delta^P\|F\|_{\mathcal{X}_{1,-\beta}^\delta},
\end{align}
where $0<\beta<b<\frac{1}{2}$, and $P=\frac{1}{2}(1-\frac{\beta}{b})>0$.
\end{lemma}

\begin{proof}
As we did in the proof of Lemma 3.7, we write
\begin{align*}
\|e^{-itH}f\|_{\mathcal{X}_{s,b}}&=\|W\la\nabla\ra^s\la\tau+\Delta\ra^bW^* \mathcal{F}_t(We^{it\Delta}W^* f)(\tau)\|_{L_{\tau,x}^2}\\
&=\|e^{it\Delta}(W^* f)\|_{X_{s,b}},\\
\Big\|\int_0^t e^{-i(t-s)H} F(s)ds\Big\|_{\mathcal{X}_{s,b}}&=\Big\|W\la\nabla\ra^s\la\tau+\Delta\ra^bW^* \mathcal{F}_t\Big(\int_0^t We^{i(t-s)\Delta}W^* F(s)ds\Big)(\tau)\Big\|_{L_{\tau,x}^2}\\
&=\Big\|\int_0^t e^{i(t-s)\Delta} (W^*F)(s)ds\Big\|_{X_{s,b}}.
\end{align*}
Then, $(3.3)$, $(3.4)$ and $(3.5)$ follow from $(3.16)$, $(3.17)$ and $(3.18)$ of \cite{CKSTT1}. For example, by $(3.17)$ of \cite{CKSTT1} and $(3.2)$, we obtain that
\begin{align*}
\Big\|\int_0^t e^{-i(t-s)H} F(s)ds\Big\|_{\mathcal{X}_{1,1/2+}^\delta}&=\Big\|\int_0^t e^{i(t-s)\Delta} W^*F(s)ds\Big\|_{X_{1, 1/2+}^\delta}\\
&\lesssim\|W^*F\|_{X_{1,-1/2+}^\delta}=\|F\|_{\mathcal{X}_{1,-1/2+}^\delta}.
\end{align*}
\end{proof}

\begin{lemma}[Bilinear estimate involving the $\mathcal{X}_{s,b}^\delta$-norms]
Suppose that $V\in B$ and $\|V_-\|_{\mathcal{K}}<4\pi$. If $\mathcal{P}_{N_1}u_1=u_1$ and $\mathcal{P}_{N_2}u_2=u_2$, then
\begin{equation}
\|u_1u_2\|_{L_{t\in[0,\delta]}^2L_x^2}\lesssim\frac{N_1}{N_2^{1/2}}\|u_1\|_{\mathcal{X}_{0,1/2+}^\delta}\|u_2\|_{\mathcal{X}_{0,1/2+}^\delta}.
\end{equation}
\end{lemma}

\begin{proof}
By the structure formula for the wave operator, we write
$$u_i(x)=(WW^*u_i)(x)=(W^*u_i)(x)+(\widetilde{W^*u_i})(x),$$
where
$$(\widetilde{W^*u_i})(x)=\int_{\mathbb{R}^3}\int_{O(3)} (W^*u_i)(sx+y) d g_{s,y}(x)dy.$$
Thus,
\begin{align*}
\|u_1u_2\|_{L_{t\in[0,\delta]}^2L_x^2}&\leq\|(W^*u_1)(W^*u_2)\|_{L_{t\in[0,\delta]}^2L_x^2}+\|(\widetilde{W^*u_1})(W^*u_2)\|_{L_{t\in[0,\delta]}^2L_x^2}\\
&+\|(W^*u_1)(\widetilde{W^*u_2})\|_{L_{t\in[0,\delta]}^2L_x^2}+\|(\widetilde{W^*u_1})(\widetilde{W^*u_2})\|_{L_{t\in[0,\delta]}^2L_x^2}.
\end{align*}
For example, consider
\begin{equation}
\|(\widetilde{W^*u_1})(W^*u_2)\|_{L_{t\in[0,\delta]}^2L_x^2}.
\end{equation}
By the Minkowski inequality, it is bounded by
$$\int_{\mathbb{R}^3}\int_{O(3)}\|(W^*u_1)(sx+y)(W^*u_2)\|_{L_{t\in[0,\delta]}^2L_x^2}d\|g_{s,y}(x)\|_{L_x^\infty}dy.$$
Observe that by the intertwining property, $W^*u_i=W^*\mathcal{P}_{N_i}u_i=P_{N_i}W^*u_i$ is localized in $|\xi|\sim N_i$ in frequency. Thus, applying the bilinear estimate in the homogeneous case (Lemma 2.1 of \cite{CKSTT1}) and $(3.2)$, we get
\begin{align*}
&\|(W^*u_1)(sx+y)W^*u_2(x)\|_{L_{t\in[0,\delta]}^2L_x^2}\lesssim\frac{N_1}{N_2^{1/2}}\|(W^*u_1)(s\cdot+y)\|_{X_{0,1/2+}^\delta}\|W^*u_2\|_{X_{0,1/2+}^\delta}\\
&=\frac{N_1}{N_2^{1/2}}\|W^*u_1\|_{X_{0,1/2+}^\delta}\|W^*u_2\|_{X_{0,1/2+}^\delta}=\frac{N_1}{N_2^{1/2}}\|u_1\|_{\mathcal{X}_{0,1/2+}^\delta}\|u_2\|_{\mathcal{X}_{0,1/2+}^\delta}.
\end{align*}
We thus conclude that 
\begin{align*}
(3.7)\lesssim\int_{\mathbb{R}^3}\int_{O(3)}\frac{N_1}{N_2^{1/2}}\|u_1\|_{\mathcal{X}_{0,1/2+}^\delta}\|u_2\|_{\mathcal{X}_{0,1/2+}^\delta}d \|g_{s,y}(x)\|_{L_x^\infty} dy\lesssim\frac{N_1}{N_2^{1/2}}\|u_1\|_{\mathcal{X}_{0,1/2+}^\delta}\|u_2\|_{\mathcal{X}_{0,1/2+}^\delta}.
\end{align*}
By the same way, we estimate other terms.
\end{proof}

\begin{remark}[Sharp constants for scaled Schr\"odinger operators]
The sharp constants for all of the above estimates are invariant under the scaling $V\mapsto V_r:=\frac{1}{r^2}V(\frac{\cdot}{r})$. For example, let $C(V)$ is the sharp constant for the Sobolev inequality (Lemma 3.6). Then, by the Stone's formula for $H^{-\frac{s}{2}}$ and scaling, it is easy to check that 
$$\|(-\Delta+V)^{-\frac{s}{2}}f(r\cdot)\|_{L^q}\leq C(V)\|f(r\cdot)\|_{L^p}\Longleftrightarrow \|(-\Delta+V_r)^{-\frac{s}{2}}f\|_{L^q}\leq C(V_r)\|f\|_{L^p},$$
which implies $C(V)=C(V_r)$.
\end{remark}

\section{Almost Conservation Law $\Rightarrow$ Global Well-posedness}

\subsection{Perturbed $I$-operator}
For large $N\gg1$, let $m_N:(0,+\infty)\to\mathbb{R}$ be a smooth non-increasing function such that
\begin{equation*}
m_N(\lambda):=\left\{\begin{aligned}
&1&&\textup{ for }\lambda\in(0,N)\\
&\sum_{M\in 2^{\mathbb{Z}}}\Big(\frac{N}{M}\Big)^{1-s}\chi_M&&\textup{ for }\lambda\in(2N,+\infty),
\end{aligned}
\right.
\end{equation*}
where $\chi_M$ is a dyadic partition of unity given in Section 1.4. Due to the technical issue in Section 5.5, we define the \textit{discrete} perturbed $I$-operators by
$$\mathcal{I}=\mathcal{I}_N:=m_N(\sqrt{H})=\sum_{M\in 2^{\mathbb{Z}}}m_N(M)\mathcal{P}_M.$$

The energy of $\mathcal{I}u$ is comparable to $\|u\|_{H^s}$ in the following sense:
\begin{lemma} Assume that $V\in B\cap L^\infty$ and $\|V_-\|_{\mathcal{K}}<4\pi$. Let $s\in(\frac{5}{6},1)$. Then, 
\begin{align}
E[\mathcal{I}u]&\lesssim N^{2-2s}(1+\|u\|_{H^s}^4),\\
\|u\|_{H^s}^2&\lesssim E[\mathcal{I}u]+\|u\|_{L^2}^2.
\end{align}
\end{lemma}

\begin{proof}
For $(4.1)$, by the spectral multiplier theorem and the norm equivalence, we obtain
$$\|H^{1/2}\mathcal{I}u\|_{L^2}^2\leq\|H^{1/2}\mathcal{P}_{\leq N}\mathcal{I}u\|_{L^2}^2+\|H^{1/2}\mathcal{P}_{\geq N}\mathcal{I}u\|_{L^2}^2\lesssim N^{2(1-s)}\|H^{\frac{s}{2}}u\|_{L^2}^2\lesssim N^{2(1-s)}\|u\|_{\dot{H}^s}^2.$$
For the nonlinear term, by the spectral multiplier theorem and the Sobolev inequality
$$\|\mathcal{I}u\|_{L^4}\leq \|u\|_{L^4}\lesssim\|u\|_{H^s}.$$
For $(4.2)$, by the norm equivalence and the spectral multiplier theorem to $H^{\frac{s-1}{2}}\mathcal{I}^{-1}$, we get
$$\|u\|_{\dot{H}^s}^2\sim\|H^{\frac{s}{2}}u\|_{L^2}^2=\|H^{\frac{s-1}{2}}\mathcal{I}^{-1}H^{1/2}\mathcal{I}u\|_{L^2}^2\lesssim\|H^{1/2}\mathcal{I}u\|_{L^2}^2\leq 2E[\mathcal{I}u].$$
\end{proof}

We define the standard $I$-operators by $I=I_N:=m_N(\sqrt{-\Delta})$. As an application of Lemma \ref{Approximation1} and \ref{Approximation2}, we prove the following approximation lemma:
\begin{lemma}[Approximation of  $\mathcal{I}$] If $V\in B\cap L^\infty$ and $\|V_-\|_{\mathcal{K}}<4\pi$, then for $\beta\in[0,1]$ and $N\gg1$, the difference $\tilde{\mathcal{I}}:=\mathcal{I}-I$ obeys
\begin{equation}
\||\nabla|^\beta\tilde{\mathcal{I}}\|_{L^p\to L^p}\lesssim N^{-2+\beta},\ 1<p<\infty.
\end{equation}
\end{lemma}

\begin{proof}
Since both $1(H)$ and $1(-\Delta)$ are identity maps on $L^2\cap L^p$, we write
\begin{align*}
\tilde{\mathcal{I}}_N=\mathcal{I}_N-I_N&=m_N(\sqrt{H})-m_N(\sqrt{-\Delta})=(1-m_N)(\sqrt{H})-(1-m_N)(\sqrt{-\Delta})\\
&=\sum_{M\in 2^{\mathbb{Z}}}((1-m_N)\chi_M)(\sqrt{H})-((1-m_N)\chi_M)(\sqrt{-\Delta})=\sum_{M\in 2^{\mathbb{Z}}}\mathfrak{D}_{(1-m_N)\chi_M}.
\end{align*}
It follows from Lemma \ref{Approximation1} and \ref{Approximation2} that 
$$\||\nabla|^\beta \mathfrak{D}_{(1-m_N)\chi_M}\|_{L^p\to L^p}\lesssim\left\{\begin{aligned}
&0&&\textup{ if }M\leq N\\
&M^{-2+\beta}&&\textup{ if }M>N.
\end{aligned}
\right.$$
Summing in $M$, we complete the proof.
\end{proof}

\subsection{Almost conservation law $\Rightarrow$ Global well-posedness}
In the next section, we will show that the energy of $\mathcal{I}u$ is almost conserved:
\begin{proposition}[Almost conservation law]\label{prop:AlmostConservation} Suppose that $V\in B\cap L^\infty$ and $\|V_-\|_{\mathcal{K}}<4\pi$. Let $s\in(\frac{5}{6},1)$. There exists a uniform $\delta\in(0,1]$ such that if $E[\mathcal{I}_Nu_0]\leq1$, then the solution $u(t)$ obeys
\begin{equation}
E[\mathcal{I}_Nu](t)=E[\mathcal{I}_Nu_0]+O(N^{-1+}),\textup{ for }t\in[0,\delta].
\end{equation}
\end{proposition}

Now we will prove the main theorem, assuming the almost conservation law.

\begin{proof}[Proof of Theorem \ref{thm:MainTheorem}, assuming Proposition \ref{prop:AlmostConservation}]
Let $r\gg 1$ to be chosen later. We denote $H_r:=-\Delta+V_r$, $V_r=\frac{1}{r^2}V(\frac{\cdot}{r})$, $\mathcal{I}_{N,r}:=m_N(\sqrt{H_r})$, $u_r(t,x):=\frac{1}{r}u(\frac{t}{r},\frac{x}{r})$ and $u_{r,0}:=\tfrac{1}{r}u_0(\tfrac{\cdot}{r})$. Then $u_r(t,x)$ solves
$$i\partial_tu_r-H_r u_r-|u_r|^2u_r=0;\ u_r(0)=u_{r,0}.$$
Observe that 
\begin{equation}\label{eq:claim1}
E_{V_r}[\mathcal{I}_{N, r}u_{r,0}]\leq C_0N^{2-2s}r^{1-2s}(1+\|u_0\|_{H^s})^4.
\end{equation}
Indeed, by the argument in the proof of Lemma 4.1, we have
\begin{align*}
\|(H_r)^{\frac{1}{2}}\mathcal{I}_{N, r}u_{r,0}\|_{L^2}^2&\lesssim N^{2-2s}\|u_{r,0}\|_{\dot{H}^s}^2=N^{2-2s}r^{1-2s}\|u_0\|_{\dot{H}^s}^2\\
\|\mathcal{I}_{N, r}u_{r,0}\|_{L^4}^4&\lesssim \|u_{r,0}\|_{L^4}^4=r^{-1}\|u_0\|_{L^4}^4\lesssim r^{-1}\|u_0\|_{H^s}^4\leq r^{1-2s}\|u_0\|_{H^s}^4.
\end{align*}
Now we choose 
$$r=\Big(\frac{1}{2C_0}\Big)^{\frac{1}{1-2s}} N^{\frac{2s-2}{1-2s}}(1+\|u_0\|_{\dot{H}^s})^{-\frac{4}{1-2s}}\Rightarrow E_{V_r}[\mathcal{I}_{N, r}u_{r,0}]\leq\frac{1}{2}\textup{ (by $(\ref{eq:claim1})$)}.$$
Then it follows from Proposition \ref{prop:AlmostConservation} that there exists $C_1>0$ such that
\begin{equation}
E_{V_r}[\mathcal{I}_{N, r}u_{r}(C_1N^{1-}\delta)]\sim1.
\end{equation}
Note that $C_1$, $\delta$ and the implicit constant in $(4.6)$ are independent of scaling. Indeed, $\|V_r\|_{B}$ is invariant and $\|V_r\|_{L^\infty}$ is uniformly bounded in $r\gg1$. As mentioned in Remark 3.10, the sharp constants for all the estimates used in the proof of Proposition \ref{prop:AlmostConservation} will not depend on scaling $V\mapsto V_r$.

Take $T_0=\frac{C_1N^{1-}\delta}{r^2}\sim N^{\frac{5-6s-}{1-2s}}$. Unscaling $u_r$, we prove that the energy of $\mathcal{I}_N u$ grows at most polynomially in time:
$$E[\mathcal{I}_Nu(T_0)]=r E_{V_r}[\mathcal{I}_{N, r}u_{r}(r^2 T_0)]\sim r\lesssim N^{\frac{2s-2}{1-2s}}\lesssim T_0^{\frac{1-s+}{3(s-\frac{5}{6})}}.$$
By $(4.2)$ and mass conservation, $\|u(t)\|_{H^s}$ also grows at most polynomially in time.
\end{proof}

\section{Proof of Almost Conservation Law: Proposition \ref{prop:AlmostConservation}}

\subsection{Time interval}
We begin by choosing a short time interval $[0,\delta]$ in Proposition \ref{prop:AlmostConservation}.
\begin{proposition}[Uniform interval]
Assume that $V\in B\cap L^\infty$, $\|V_-\|_{\mathcal{K}}<4\pi$ and $\frac{5}{6}<s<1$. Then there exists $\delta>0$ such that if $E[\mathcal{I}u_0]\leq1$, the solution $u$ to $\textup{NLS}_V$ with initial data $u_0$ exists on the interval $[0,\delta]$ and it obeys 
\begin{equation}
\|H^{1/2}\mathcal{I}u\|_{\mathcal{X}_{0,1/2+}^\delta}\lesssim1.
\end{equation} 
\end{proposition}

For the proof, we need the following lemma. For notational convenience, we omit the time interval $[0,\delta]$ in the norm $\|\cdot\|_{L_{t\in[0,\delta]}^p}$ if there is no confusion.

\begin{lemma}
Assume that $V\in B\cap L^\infty$, $\|V_-\|_{\mathcal{K}}<4\pi$ and $\frac{5}{6}<s<1$. Then, for $\delta>0$,
\begin{equation}
\|u\|_{L_{t\in[0,\delta]}^{12}L_x^6}\lesssim \|H^{1/2}\mathcal{I}u\|_{\mathcal{X}_{0,1/2+}^\delta}.
\end{equation}
\end{lemma}

\begin{proof}
Split $u=\mathcal{P}_{\leq N}u+\mathcal{P}_{>N}u$. By the Sobolev inequality, Strichartz estimates and the spectral multiplier theorem, we obtain
\begin{align*}
\|\mathcal{P}_{\leq N}u\|_{L_t^{12}L_x^6}&\leq\delta^{1/12}\|\mathcal{P}_{\leq N}u\|_{L_t^\infty L_x^6}\lesssim\|H^{1/2}\mathcal{I}u\|_{L_t^\infty L_x^2}\lesssim\|H^{1/2}\mathcal{I}u\|_{\mathcal{X}_{0,1/2+}^\delta},\\
\|\mathcal{P}_{\geq N}u\|_{L_t^{12}L_x^6}&\leq\delta^{\frac{6s-5}{12}}\|\mathcal{P}_{\geq N}u\|_{L_t^{2/(1-s)}L_x^6}\lesssim\|(H^{-\frac{1-s}{2}}\mathcal{I}^{-1})\mathcal{P}_{\geq N}(H^{1/2}\mathcal{I}u)\|_{L_t^{2/(1-s)}L_x^{6/(1+2s)}}\\
&\lesssim\|H^{1/2}\mathcal{I}u\|_{L_t^{2/(1-s)}L_x^{6/(1+2s)}}\lesssim \|H^{1/2}\mathcal{I}u\|_{\mathcal{X}_{0,1/2+}^\delta}.
\end{align*}
\end{proof}

\begin{proof}[Proof of Proposition 5.1]
Applying Lemma 3.8 to the Duhamel formula
$$u(t)=e^{-itH}u_0+i\int_0^te^{-i(t-s)H}(|u|^2u)(s) ds,$$
we write
$$\|H^{1/2}\mathcal{I}u\|_{\mathcal{X}_{0,1/2+}^\delta}\lesssim\|H^{1/2} \mathcal{I}u_0\|_{L^2}+\|H^{1/2}\mathcal{I}(|u|^2u)\|_{\mathcal{X}_{0,-1/2+}^\delta}\lesssim1+\delta^{0+}\|H^{1/2}\mathcal{I}(|u|^2u)\|_{\mathcal{X}_{0,-1/4}^\delta}.$$
Then a small constant $\delta>0$ will be obtained by the standard nonlinear iteration argument, once we show that
\begin{equation}
\|H^{1/2}\mathcal{I}(|u|^2u)\|_{\mathcal{X}_{0,-1/4}^\delta}\lesssim\|H^{1/2}\mathcal{I}u\|_{\mathcal{X}_{0,1/2+}^\delta}^3.
\end{equation}
Observe that by interpolation between $\|v\|_{L_{t\in\mathbb{R}}^2L_x^2}\leq\|v\|_{\mathcal{X}_{0,0}}$ and $\|v\|_{L_{t\in\mathbb{R}}^\infty L_x^2}\leq\|v\|_{\mathcal{X}_{0,1/2+}}$,
$$\|v\|_{L_t^3L_x^2(\mathbb{R}\times\mathbb{R}^3)}\lesssim\|v\|_{\mathcal{X}_{0,1/4}}\Leftrightarrow \|v\|_{\mathcal{X}_{0,-1/4}} \lesssim\|v\|_{L_t^{3/2}L_x^2(\mathbb{R}\times\mathbb{R}^3)}.$$
Hence, it follows that 
$$\|w\|_{\mathcal{X}_{0,-1/4}^\delta}=\inf_{v=w\textup{ on }[0,\delta]}\|v\|_{\mathcal{X}_{0,-1/4}}\leq\inf_{v=w\textup{ on }[0,\delta]}\|v\|_{L_t^{3/2}L_x^2(\mathbb{R}\times\mathbb{R}^3)}=\|w\|_{L_{t\in[0,\delta]}^{3/2}L_x^2}.$$
Therefore, for $(5.3)$, it is enough to show that 
$$\|H^{1/2}\mathcal{I}(|u|^2u)\|_{L_t^{3/2}L_x^2} \lesssim\|H^{1/2}\mathcal{I}u\|_{\mathcal{X}_{0,1/2+}^\delta}^3.$$
For the low frequency part, by the spectral multiplier theorem and $(5.2)$,
$$\|\mathcal{P}_{\leq 1}H^{1/2}\mathcal{I}(|u|^2u)\|_{L_t^{3/2}L_x^2}\lesssim\||u|^2u\|_{L_t^{3/2}L_x^2}\leq\delta^{5/12}\|u\|_{L_t^{12}L_x^6}^3\lesssim\|H^{1/2}\mathcal{I}u\|_{\mathcal{X}_{0,1/2+}^\delta}^3.$$
For the high frequency part, by Lemma 2.1,
\begin{align*}
&\|\mathcal{P}_{>1}H^{1/2}\mathcal{I}(|u|^2u)-P_{>1}|\nabla|I(|u|^2u)\|_{L_t^{3/2}L_x^2}\\
&\leq\sum_{N\geq1}\|(\mathcal{P}_NH^{1/2}\mathcal{I}-P_N|\nabla|I)(|u|^2u)\|_{L_t^{3/2}L_x^2}\lesssim \sum_{N\geq1}N^{-1}\||u|^2u\|_{L_t^{3/2}L_x^2}\\
&\leq\||u|^2u\|_{L_t^{3/2}L_x^2}\leq\delta^{5/12}\|u\|_{L_t^{12}L_x^6}^3\lesssim\|H^{1/2}\mathcal{I}u\|_{\mathcal{X}_{0,1/2+}^\delta}^3.
\end{align*}
Thus it remains to show 
$$\||\nabla|I(|u|^2u)\|_{L_t^{3/2}L_x^2} \lesssim\|H^{1/2}\mathcal{I}u\|_{\mathcal{X}_{0,1/2+}^\delta}^3.$$
We split each $u$ into the low and high frequency parts, and then apply the Leibniz rule for $|\nabla|I$. When $|\nabla|I$ hits the low frequency, it is bounded by
$$\delta^{5/12}\||\nabla|IP_{\leq 1}u\|_{L_t^{12}L_x^6}\|\tilde{u}\|_{L_t^{12}L_x^6}\|\tilde{u}\|_{L_t^{12}L_x^6} \lesssim\|u\|_{L_t^{12}L_x^6}^3\lesssim\|H^{1/2}\mathcal{I}u\|_{\mathcal{X}_{0,1/2+}^\delta}^3\textup{ (by $(5.2)$)}$$
where $\tilde{u}$ is either $P_{\leq 1}u$ or $P_{>1}u$. When $|\nabla|I$ hits the high frequency, it is bounded by
$$\||\nabla|IP_{>1}u\|_{L_t^2L_x^6}\|\tilde{u}\|_{L_t^{12}L_x^6}\|\tilde{u}\|_{L_t^{12}L_x^6}\lesssim\||\nabla|IP_{>1}u\|_{L_t^2L_x^6}\|H^{1/2}\mathcal{I}u\|_{\mathcal{X}_{0,1/2+}^\delta}^2.$$
For $|\nabla|IP_{>1}u$, we write
$$\||\nabla|IP_{>1}u\|_{L_t^2L_x^6}\leq\|\mathcal{P}_{>1}H^{1/2}\mathcal{I}u\|_{L_t^2L_x^6}+\|(H^{1/2}\mathcal{I}\mathcal{P}_{>1}-|\nabla|IP_{>1})u\|_{L_t^2L_x^6}.$$
Applying Strichartz estimates to the first term and Lemma 2.1 and $(5.2)$ to the second term, we bound $\||\nabla|IP_{>1}u\|_{L_t^2L_x^6}$ by $\sim\|H^{1/2}\mathcal{I}u\|_{\mathcal{X}_{0,1/2+}^\delta}$.
\end{proof}

\subsection{A priori estimates}
Before proving the proposition, we collect a priori estimates. Let $\delta>0$ be in Proposition 5.1. We claim that we may assume that 
\begin{equation}
\|\mathcal{I}u\|_{L_{t\in[0,\delta]}^\infty L_x^4}^4\leq16.
\end{equation}
Indeed, by $H^s$-continuity of $u(t)$ and $(4.2)$, one can find a subinterval $[0,\tilde{\delta}]$ where $(5.4)$ holds. Then, it follows from the argument of Section 5.4$\sim$5.6 with $(5.1)$ and $(5.4)$ on $[0,\tilde{\delta}]$ that $$E[\mathcal{I}u(t)]\leq 2\textup{ for }0\leq t\leq\tilde{\delta}\Rightarrow \|\mathcal{I}u(t)\|_{L_x^4}^4\leq8\textup{ for }0\leq t\leq\tilde{\delta}.$$
Thus one can extend $[0,\tilde{\delta}]$ little bit more with the same a priori bounds. Repeating, we extend to $[0,\delta]$. For simplicity, we omit this iteration procedure and just assume $(5.4)$.
\begin{lemma}[Collection of a priori estimates] Assume that $V\in B\cap L^\infty$, $\|V_-\|_{\mathcal{K}}<4\pi$ and $u$ satisfies $(5.1)$ and $(5.4)$. Then the following estimates hold:
\begin{equation}\label{APriori1}
\|H^{1/2}\mathcal{I}u\|_{L_t^q L_x^r}\lesssim 1,\ \|\nabla\mathcal{I}u\|_{L_t^q L_x^r}\lesssim 1\textup{, where }\tfrac{2}{q}+\tfrac{3}{r}=\tfrac{3}{2},\ 2\leq q,r\leq\infty,
\end{equation}
\begin{equation}\label{APriori2}
\|u\|_{L_t^{12}L_x^6}\lesssim 1\textup{ (Lemma 5.2)},
\end{equation}
\begin{equation}\label{APriori3}
\||\nabla|^\beta\tilde{\mathcal{I}}u\|_{L_t^{12}L_x^6}\lesssim N^{-2+\beta}\textup{ (Lemma 4.2 and $(5.7)$)},
\end{equation}
\begin{equation}\label{APriori4}
\|I^{-1}\mathcal{I}u\|_{L_t^4L_x^6}\lesssim1,
\end{equation}
\begin{equation}\label{APriori5}
\|I^{-1}\tilde{\mathcal{I}}u\|_{L_t^{12}L_x^6}\lesssim N^{-2}.
\end{equation}
\end{lemma}

\begin{proof}
$(5.5)$: The first inequality follows from Strichartz estimates. For the second one, we observe that by the norm equivalence, 
$$\|\nabla\mathcal{I}u\|_{L_t^\infty L_x^2}\sim \|H^{1/2}\mathcal{I}u\|_{L_t^\infty L_x^2}\lesssim 1.$$
By interpolation, it suffices to show for $(q,r)=(2,6)$. Indeed, for the low frequency part, by Lemma 5.2, 
$$\|\nabla P_{\leq 1}\mathcal{I}u\|_{L_t^2L_x^6}\lesssim \delta^{5/12}\|u\|_{L_t^{12}L_x^6}\lesssim1.$$
For the high frequency part, we write 
$|\nabla|P_{> 1}\mathcal{I}u=(|\nabla|P_{>1}-H^{1/2}\mathcal{P}_{>1})\mathcal{I}u+H^{1/2}\mathcal{P}_{>1}\mathcal{I}u$. For the first term, by Proposition 2.2, we obtain 
\begin{align*}
\|(|\nabla|P_{>1}-H^{1/2}\mathcal{P}_{>1})\mathcal{I}u\|_{L_t^2L_x^6}&\leq\delta^{5/12}\sum_{M>1}\|(|\nabla|P_M-H^{1/2}\mathcal{P}_M)\mathcal{I}u\|_{L_t^{12}L_x^6}\\
&\leq\sum_{M>1}M^{-1}\|u\|_{L_t^{12}L_x^6}\leq \|u\|_{L_t^{12}L_x^6}\lesssim1.
\end{align*}
The second term is bounded by Strichartz estimates.\\
$(5.8)$: Similarly, for the low and high frequency parts, we have
\begin{align*}
\|P_{\leq N}I^{-1}\mathcal{I}u\|_{L_t^4L_x^6}&\leq\|P_{\leq N}\mathcal{I}u\|_{L_t^4L_x^6}\leq\delta^{1/6}\|u\|_{L_t^{12}L_x^6}\lesssim1,\\
\|P_{>N}I^{-1}\mathcal{I}u\|_{L_t^4L_x^6}&\leq\||\nabla|^{1-s}\mathcal{I}u\|_{L_t^4L_x^6}\lesssim\delta^{\frac{2s-1}{4}}\|\nabla\mathcal{I}u\|_{L_t^{2/(1-s)}L_x^{6/(1+2s)}}\lesssim1.
\end{align*}
$(5.9)$: Splitting $I^{-1}\tilde{\mathcal{I}}u$ into the low and high frequency parts, we write
\begin{align*}
\|I^{-1}\tilde{\mathcal{I}}u\|_{L_t^4L_x^6}&\leq\delta^{1/6}\|I^{-1}\tilde{\mathcal{I}}u\|_{L_t^{12}L_x^6}\leq\|P_{\leq N}I^{-1}\tilde{\mathcal{I}}u\|_{L_t^{12}L_x^6}+\|P_{>N}I^{-1}\tilde{\mathcal{I}}u\|_{L_t^{12}L_x^6}\\
&\lesssim\|\tilde{\mathcal{I}}u\|_{L_t^{12}L_x^6}+N^{-(1-s)}\||\nabla|^{1-s}\tilde{\mathcal{I}}u\|_{L_t^{12}L_x^6}\lesssim N^{-2}.
\end{align*}
\end{proof}

\subsection{Outline of the proof}
Let $\delta>0$ be a small number given by Proposition 5.1, and choose a large number $N\gg 1/\delta$. By the persistence of regularity \cite[Corollary 4.4]{H2}, it suffices to show the almost conservation laws for solutions in $C_t([0,\delta]; \mathfrak{H}_x^2)$, where $\mathfrak{H}^2$ is the domain of the self-adjoint operator $H$. It will guarantee the following formal calculations make sense.

By the fundamental theorem of calculus,
\begin{align*}
E[\mathcal{I}u(\delta)]-E[\mathcal{I}u_0]&=\int_0^\delta\frac{d}{dt}\Big(\frac{1}{2}\int (H\mathcal{I}u)\overline{\mathcal{I}u} dx+\frac{1}{4}\int_{\mathbb{R}^3}|\mathcal{I}u|^4 dx\Big)dt\\
&=\int_0^\delta\int_{\mathbb{R}^3}\Re[\overline{\mathcal{I}u_t} (H\mathcal{I}u+|\mathcal{I}u|^2 \mathcal{I}u)]dxdt.
\end{align*}
Plugging the identity 
\begin{align*}
\Re[\overline{\mathcal{I}u_t} (H\mathcal{I}u+|\mathcal{I}u|^2 \mathcal{I}u)]&=\Re[\overline{\mathcal{I}u_t} (H\mathcal{I}u+|\mathcal{I}u|^2 \mathcal{I}u-i\mathcal{I}u_t)]\\
&=\Re[\overline{\mathcal{I}u_t} (H\mathcal{I}u+|\mathcal{I}u|^2 \mathcal{I}u-\mathcal{I}(Hu+|u|^2u))]\\
&=\Im[\overline{\mathcal{I}(Hu+|u|^2u)} (|\mathcal{I}u|^2 \mathcal{I}u-\mathcal{I}(|u|^2u))],
\end{align*}
we write
$$E[\mathcal{I}u(\delta)]-E[\mathcal{I}u_0]=\Im\int_0^\delta\int_{\mathbb{R}^3}\overline{\mathcal{I}(Hu+|u|^2u)} (|\mathcal{I}u|^2 \mathcal{I}u-\mathcal{I}(|u|^2u)) dxdt.$$
Hence the almost conservation law follows once we show that 
\begin{align*}
\textup{Term}_1&=\int_0^\delta\int_{\mathbb{R}^3}\overline{H\mathcal{I}u} (|\mathcal{I}u|^2 \mathcal{I}u-\mathcal{I}(|u|^2u)) dxdt=O(N^{-1+}),\\
\textup{Term}_2&=\int_0^\delta\int_{\mathbb{R}^3}\overline{\mathcal{I}(|u|^2u)} (|\mathcal{I}u|^2 \mathcal{I}u-\mathcal{I}(|u|^2u)) dxdt=O(N^{-1+}).
\end{align*}
We will prove them in two steps. First, we approximate $\textup{Term}_1$ and $\textup{Term}_2$ by
\begin{align*}
\textup{Term}_1'&=\int_0^\delta\int_{\mathbb{R}^3}\overline{H\mathcal{I}u}\Big(|\mathcal{I}u|^2 \mathcal{I}u-I(|I^{-1}\mathcal{I}u|^2(I^{-1}\mathcal{I}u)\Big) dxdt,\\
\textup{Term}_2'&=\int_0^\delta\int_{\mathbb{R}^3}\overline{\mathcal{I}(|u|^2u)}\Big(|\mathcal{I}u|^2 \mathcal{I}u-I(|I^{-1}\mathcal{I}u|^2(I^{-1}\mathcal{I}u)\Big) dxdt
\end{align*}
with $O(N^{-1+})$-error. Next, we show that $\textup{Term}_1'$ and $\textup{Term}_2'$ are $O(N^{-1+})$.

\subsection{Approximation step} 
We will show that $(\textup{Term}_1-\textup{Term}_1')=O(N^{-1})$. First, we write 
$$(\textup{Term}_1-\textup{Term}_1')\leq\|H\mathcal{I}u\|_{L_t^\infty\dot{H}_x^{-1}}\|\mathcal{I}(|u|^2u)-I(|I^{-1}\mathcal{I}u|^2(I^{-1}\mathcal{I}u)\|_{L_t^1H_x^1}.$$
For the first term, by the norm equivalence and $(5.5)$, $\|H\mathcal{I}u\|_{L_t^\infty\dot{H}_x^{-1}}\sim\|\nabla\mathcal{I}u\|_{L_t^\infty L_x^2}\lesssim1$. For the second term, we split
$$\mathcal{I}(|u|^2u)-I(|I^{-1}\mathcal{I}u|^2(I^{-1}\mathcal{I}u)=\tilde{\mathcal{I}}(|u|^2u)-I(|I^{-1}\mathcal{I}u|^2(I^{-1}\mathcal{I}u)-|u|^2u).$$
By Lemma 4.2 and $(5.7)$,  
$$\|\tilde{\mathcal{I}}(|u|^2u)\|_{L_t^1\dot{H}_x^1}\lesssim N^{-1}\||u|^2u\|_{L_t^1L_x^2}\lesssim \delta^{3/4}N^{-1}\|u\|_{L_t^{12}L_x^6}^3\lesssim N^{-1}.$$
It remains to show that
$$\||\nabla|I(|I^{-1}\mathcal{I}u|^2(I^{-1}\mathcal{I}u)-|u|^2u)\|_{L_x^1L_x^2}\lesssim N^{-1}.$$
We split 
$$|I^{-1}\mathcal{I}u|^2(I^{-1}\mathcal{I}u)-|u|^2u=(I^{-1}\tilde{\mathcal{I}}u)|I^{-1}\mathcal{I}u|^2+u\overline{(I^{-1}\tilde{\mathcal{I}}u)}(I^{-1}\mathcal{I}u)+|u|^2(I^{-1}\tilde{\mathcal{I}}u).$$
Then, by the H\"older inequalities, the Leibniz rule for $|\nabla|I$ and Lemma 5.3, we obtain
\begin{align*}
&\||\nabla|I((I^{-1}\tilde{\mathcal{I}}u)|I^{-1}\mathcal{I}u|^2)\|_{L_t^1L_x^2}\lesssim\delta^{5/12}\|\nabla\tilde{\mathcal{I}}u\|_{L_t^{12}L_x^6}\|I^{-1}\mathcal{I}u\|_{L_t^4L_x^6}^2\\
& \ \ \ \ \ \ \ \ \ \ \ \ \ \ \ \ \ \ \ \ \ \ \ \ \ \ \ \ \ \ \ \ \ \ \ \ \ \ +\delta^{1/6}\|I^{-1}\tilde{\mathcal{I}}u\|_{L_t^{12}L_x^6}\||\nabla|\mathcal{I}u\|_{L_t^2L_x^6}\|I^{-1}\mathcal{I}u\|_{L_t^4L_x^6}\lesssim N^{-1};\\
&\||\nabla|I(u(\overline{I^{-1}\tilde{\mathcal{I}}u})(I^{-1}\mathcal{I}u))\|_{L_t^1L_x^2}\lesssim\delta^{1/6}\||\nabla|Iu\|_{L_t^2L_x^6}\|I^{-1}\tilde{\mathcal{I}}u\|_{L_t^{12}L_x^6}\|I^{-1}\mathcal{I}u\|_{L_t^4L_x^6}\\
&\ \ \ \ \ \ \ \ \ \ \ \ \ \ \ \ \ \ \ \ \ \ \ \ \ \ \ \ \ \ \ \ \ \ \ \ \ \ \ +\delta^{7/12}\|u\|_{L_t^{12}L_x^6}\||\nabla|\tilde{\mathcal{I}}u\|_{L_t^{12}L_x^6}\|I^{-1}\mathcal{I}u\|_{L_t^4L_x^6}\\
&\ \ \ \ \ \ \ \ \ \ \ \ \ \ \ \ \ \ \ \ \ \ \ \ \ \ \ \ \ \ \ \ \ \ \ \ \ \ \ +\delta^{1/3}\|u\|_{L_t^{12}L_x^6}\|I^{-1}\tilde{\mathcal{I}}u\|_{L_t^{12}L_x^6}\||\nabla|\mathcal{I}u\|_{L_t^2L_x^6}\lesssim N^{-1};\\
&\||\nabla|I(|u|^2(I^{-1}\tilde{\mathcal{I}}u))\|_{L_t^1L_x^2}\lesssim\delta^{1/3}\||\nabla|Iu\|_{L_t^2L_x^6}\|u\|_{L_t^{12}L_x^6}\|I^{-1}\tilde{\mathcal{I}}u\|_{L_t^{12}L_x^6}\\
& \ \ \ \ \ \ \ \ \ \ \ \ \ \ \ \ \ \ \ \ \ \ \ \ \ \ \ \ \ \ \ \ +\delta^{1/3}\|u\|_{L_t^{12}L_x^6}^2\|\nabla\tilde{\mathcal{I}}u\|_{L_t^2L_x^6}\lesssim N^{-1}.
\end{align*}
Collecting all, we conclude that $(\textup{Term}_1-\textup{Term}_1')=O(N^{-1})$. Similarly, one can show that $(\textup{Term}_2-\textup{Term}_2')=O(N^{-1})$.

\subsection{$\textup{Term}_1'$}
We will show that $\textup{Term}_1'=O(N^{-1+})$. The proof will closely follow from that in \cite{CKSTT1}, but it has to be modified for the following technical reasons. Observe that by Fourier transform, we write  
\begin{equation}
\textup{Term}_1'=\int_0^\delta\int_{\sum_{j=1}^4\xi_j=0} \Big(1-\frac{m(\xi_1)}{m(\xi_2)m(\xi_3)m(\xi_4)}\Big)\widehat{\overline{H\mathcal{I}u}}(\xi_1)\widehat{\mathcal{I}u}(\xi_2) \widehat{\overline{\mathcal{I}u}}(\xi_3)\widehat{\mathcal{I}u}(\xi_4).
\end{equation}
In our case, we cannot take out the symbol $(1-\frac{m(\xi_1)}{m(\xi_2)m(\xi_3)m(\xi_4)})$ from the integral of each dyadic piece as in the homogeneous case \cite{CKSTT1}, since the $\mathcal{X}_{s,b}$-norm is not defined by Fourier transform. To solve this problem, we will discretize the symbol. The Strichartz exponents in \cite{CKSTT1} also have to be modified, because the norm equivalence $\|\nabla u\|_{L^r}\sim\|H^{1/2}u\|_{L^r}$ is valid only for $1<r<3$. 

We introduce a new Littlewood-Paley projection $Q_M:=(\tilde{\chi}_M\hat{f})^\vee$ where 
$$\tilde{\chi}_M(\lambda):=\frac{m(M)\chi_M(\lambda)}{\sum_{K\in 2^{\mathbb{Z}}}m(K)\chi_K(\lambda)}.$$
Observe that 
\begin{align*}
\frac{1}{m}&=\frac{1}{\sum_K m(K)\chi_K}=\sum_M\frac{\chi_M}{\sum_Km(K)\chi_K}=\sum_M\frac{1}{m(M)}\frac{m(M)\chi_M}{\sum_Km(K)\chi_K}=\sum_M\frac{\tilde{\chi}_M}{m(M)},
\end{align*}
and thus
$$I^{-1}f=\sum_M \frac{1}{m(M)}Q_Mf.$$
Using two Littlewood-Paley projections $P_M$ and $Q_M$, we decompose $(5.10)$ into the sum of dyadic pieces
\begin{equation}\label{DyadicPiece}
\begin{aligned}
\Big(1-\frac{m(N_1)}{m(N_2)m(N_3)m(N_4)}\Big)\int_0^\delta\int_{\mathbb{R}^3}\overline{v_1}v_2\overline{v_3}v_4 dxdt,
\end{aligned}
\end{equation}
where $v_1=P_{N_1}(H\mathcal{I}u)$ and $v_i=Q_{N_i}\mathcal{I}u$ for $i=2,3,4$. By symmetry, we may restrict the case $N_2\geq N_3\geq N_4$ in the sum. Note also that $N_1\lesssim N_2$ in $(5.11)$, since
$$\int_{\mathbb{R}^3}\overline{v_1}v_2\overline{v_3}v_4 dx=\int_{\sum_{i=1}^4\xi_i=0}\widehat{\overline{v_1}}(\xi_1)\hat{v}_2(\xi_2)\widehat{\overline{v_3}}(\xi_3)\hat{v}_4(\xi_4).$$
We split the sum into the following sub-cases.\\
\textbf{($\textup{Term}_1'$, Case 1: $N_2\ll N$)} In this case, the symbol is zero, and therefore $(\ref{DyadicPiece})=0$.\\
\textbf{($\textup{Term}_1'$, Case 2: $N_2\gtrsim N\gg N_3\geq N_4\Rightarrow N_1\sim N_2$)} By the mean value theorem,
\begin{align*}
\Big|1-\frac{m(N_1)}{m(N_2)m(N_3)m(N_4)}\Big|&=\frac{|m(N_2)-m(N_1)|}{m(N_2)}\lesssim\frac{|\nabla m(N_2)|}{m(N_2)}|N_2-N_1|\lesssim\frac{N_3}{N_2}.
\end{align*}
Using this bound and the the H\"older inequality, we write
$$|(\ref{DyadicPiece})|\lesssim\frac{N_3}{N_2}\|v_1\|_{L_t^{24/5}L_x^{36/13}}\|v_2\|_{L_t^{8/3-}L_x^4}\|v_3\|_{L_t^{24/5}L_x^{36/13}}\|v_4\|_{L_t^{24/5+}L_x^{36}}.$$
Note that $(\frac{24}{5},\frac{36}{13})$ and $(\frac{8}{3},4)$ are Strichartz exponent pairs. Thus, the norm equivalence ($\frac{36}{13}<3$) and Lemma 5.3,
\begin{equation}
\begin{aligned}
&\|v_1\|_{L_t^{24/5}L_x^{36/13}}\lesssim N_1\||\nabla|^{-1}H\mathcal{I}u\|_{L_t^{24/5}L_x^{36/13}}\sim N_1\|H^{1/2}\mathcal{I}u\|_{L_t^{24/5}L_x^{36/13}}\lesssim N_1\sim N_2;\\
&\|v_2\|_{L_t^{8/3-}L_x^4}\sim \delta^{0+}\|v_2\|_{L_t^{8/3}L_x^{4}}\leq\delta^{0+}N_2^{-1}\|\nabla\mathcal{I}u\|_{L_t^{8/3}L_x^4}\lesssim \delta^{0+}N_2^{-1};\\
&\|v_3\|_{L_t^{24/5}L_x^{36/13}}\sim N_3^{-1}\|\nabla\mathcal{I}u\|_{L_t^{24/5}L_x^{36/13}}\sim N_3^{-1};\\
&\|v_4\|_{L_t^{24/5+}L_x^{36}}\lesssim \||\nabla|^{1+}v_4\|_{L_t^{24/5+}L_x^{36/13-}}\lesssim N_4^{0+}\|\nabla\mathcal{I}u\|_{L_t^{24/5+}L_x^{36/13-}}\lesssim N_4^{0+},
\end{aligned}
\end{equation}
where $\frac{36}{13}-$ is chosen so that $(\frac{24}{5}+,\frac{36}{13}-)$ is a Strichartz exponent pair.  Therefore, we obtain
$$|(\ref{DyadicPiece})|\lesssim\delta^{0+}N_2^{-1}N_4^{0+}.$$
Summing up, we prove that  $\sum_{\textup{Case 2}}|(\ref{DyadicPiece})|\lesssim \delta^{0+}N^{-1+}.$\\
\textbf{($\textup{Term}_1'$, Case 3: $N_2\geq N_3\gtrsim N$)} 
Now, we use the trivial bound
\begin{equation}
\Big|1-\frac{m(N_1)}{m(N_2)m(N_3)m(N_4)}\Big|\lesssim\frac{m(N_1)}{m(N_2)m(N_3)m(N_4)}
\end{equation}
for the symbol, and consider the following six sub-cases separately.\\
\textbf{(Case 3-1a: $N_1\sim N_2\geq N_3\gtrsim N$; $N_4\leq N$)} Similarly, bounding $v_1, v_3$ in $L_t^{24/5}L_x^{36/13}$, $v_2$ in $L_t^{8/3-}L_x^4$ and $v_4$ in $L_t^{24/5-}L_x^{36}$ with $(5.12)$, we write
$$|(\ref{DyadicPiece})|\lesssim\frac{N_3^{1-s}}{N^{1-s}}\frac{N_1N_4^{0+}}{N_2N_3}\delta^{0+}\||\nabla|^{-1}v_1\|_{L_t^{24/5}L_x^{36/13}}\|\nabla v_2\|_{L_t^{8/3}L_x^4}.$$
Summing in $N_3, N_4$, we obtain
$$\sum_{\textup{Case 3-1a}}|(\ref{DyadicPiece})|\lesssim\delta^{0+}N^{-1+}\sum_{N_1\sim N_2\gtrsim N}\||\nabla|^{-1}v_1\|_{L_t^{24/5}L_x^{36/13}}\|\nabla v_2\|_{L_t^{8/3}L_x^4}.$$
By Corollary 2.3, Lemma \ref{Approximation1} and $(5.6)$, we approximate $|\nabla|^{-1}v_1=P_{N_1}|\nabla|^{-1}H\mathcal{I}u$ by $\mathcal{P}_{N_1}H^{1/2}\mathcal{I}u$ with error $N_1^{-2}\|H^{1/2}\mathcal{I}u\|_{L_t^{24/5}L_x^{36/13}}\lesssim N_1^{-2}$, and approximate 
$Q_{N_2}|\nabla|\mathcal{I}u$ by $\mathcal{Q}_{N_2}H^{1/2}\mathcal{I}u$ with error $N_2^{-1}\|\mathcal{I}u\|_{L_t^{8/3}L_x^4}\leq \delta^{3/8}N_2^{-1}\|\mathcal{I}u\|_{L_t^\infty L_x^4}\lesssim N_2^{-1}$, where $\mathcal{Q}_M=\tilde{\chi}_M(\sqrt{H})$. Therefore, by Strichartz estimates, we obtain
\begin{align*}
\sum_{\textup{Case 3-1a}}|(\ref{DyadicPiece})|&\lesssim\delta^{0+}N^{-1+}\sum_{N_1\sim N_2}\|\mathcal{P}_{N_1}H^{1/2}\mathcal{I}u\|_{L_t^{24/5}L_x^{36/13}}\|\mathcal{Q}_{N_2}H^{1/2}\mathcal{I}u\|_{L_t^{8/3}L_x^4}+\delta^{0+}N^{-1+}\\
&\lesssim\delta^{0+}N^{-1+}\sum_{N_1\sim N_2}\|\mathcal{P}_{N_1}H^{1/2}\mathcal{I}u\|_{\mathcal{X}_{0,1/2+}^\delta}\|\mathcal{Q}_{N_2}H^{1/2}\mathcal{I}u\|_{\mathcal{X}_{0,1/2+}^\delta}+\delta^{0+}N^{-1+}.
\end{align*}
Finally, using the Cauchy-Schwartz inequality to sum in $N_1, N_2$, we conclude that
$$\sum_{\textup{Case 3-1a}}|(\ref{DyadicPiece})|\lesssim\delta^{0+}N^{-1+}.$$
\textbf{(Case 3-1b: $N_1\sim N_2\geq N_3\gtrsim N$; $N_4\geq N$)} Bounding $v_1, v_3$ in $L_t^{24/5}L_x^{36/13}$, $v_2$ in $L_t^{8/3}L_x^4$ and $v_4$ in $L_t^{24/5}L_x^{36}$ with $(5.12)$, we get
\begin{align*}
|(\ref{DyadicPiece})|&\lesssim\frac{N_3^{1-s}N_4^{1-s}}{N^{2(1-s)}}\frac{N_1}{N_2N_3}\||\nabla|^{-1}v_1\|_{L_t^{24/5}L_x^{36/13}}\|\nabla v_2\|_{L_t^{8/3}L_x^4}\\
&= N^{-2(1-s)}N_3^{-s}N_4^{1-s}\||\nabla|^{-1}v_1\|_{L_t^{24/5}L_x^{36/13}}\|\nabla v_2\|_{L_t^{8/3}L_x^4}.
\end{align*}
Summing in $N_3,N_4$ and using the Cauchy-Schwartz inequality to sum in $N_1, N_2$ as we did in Case 3-1a, we obtain $\sum_{\textup{Case 3-1b}}|(\ref{DyadicPiece})|\lesssim N^{-1}.$\\
\textbf{(Case 3-2a: $N_2\sim N_3\gtrsim N$; $N_1\leq N$, $N_4\leq N$)} Bounding $v_1, v_3$ in $L_t^{24/5}L_x^{36/13}$, $v_2$ in $L_t^{8/3-}L_x^4$ and $v_4$ in $L_t^{24/5+}L_x^{36}$ with $(5.12)$, we get
$$|(\ref{DyadicPiece})|\lesssim\frac{N_2^{1-s}N_3^{1-s}}{N^{2(1-s)}}\frac{N_1N_4^{0+}}{N_2N_3}\delta^{0+}\lesssim\delta^{0+}N^{-1+2s}N_2^{-2s}N_4^{0+}.$$
\textbf{(Case 3-2b: $N_2\sim N_3\gtrsim N$; $N_1\geq N$, $N_4\leq N$)} Bounding $v_1, v_3$ in $L_t^{24/5}L_x^{36/13}$, $v_2$ in $L_t^{8/3-}L_x^4$ and $v_4$ in $L_t^{24/5+}L_x^{36}$ with $(5.12)$, we get
$$|(\ref{DyadicPiece})|\lesssim\frac{N_2^{1-s}N_3^{1-s}}{N^{1-s}N_1^{1-s}}\frac{N_1N_4^{0+}}{N_2N_3}\delta^{0+}\lesssim\delta^{0+}N^{-(1-s)}N_2^{-s}N_4^{0+}.$$
\textbf{(Case 3-2c: $N_2\sim N_3\gtrsim N$; $N_1\leq N$, $N_4\geq N$)} Bounding $v_1, v_3$ in $L_t^{24/5}L_x^{36/13}$, $v_2$ in $L_t^{8/3}L_x^4$ and $v_4$ in $L_t^{24/5}L_x^{36}$ with $(5.12)$, we get
$$|(\ref{DyadicPiece})|\lesssim\frac{N_2^{1-s}N_3^{1-s}N_4^{1-s}}{N^{3(1-s)}}\frac{N_1}{N_2N_3}\lesssim N^{-3(1-s)}N_2^{1-2s}N_4^{1-s}.$$
\textbf{(Case 3-2d: $N_2\sim N_3\gtrsim N$; $N_1,N_4\geq N$).} Bounding $v_1, v_3$ in $L_t^{24/5}L_x^{36/13}$, $v_2$ in $L_t^{8/3}L_x^4$ and $v_4$ in $L_t^{24/5}L_x^{36}$ with $(5.12)$, we get
$$|(\ref{DyadicPiece})|\lesssim\frac{N_2^{1-s}N_3^{1-s}N_4^{1-s}}{N_1^{1-s}N^{2(1-s)}}\frac{N_1}{N_2N_3}\lesssim N^{-2(1-s)}N_2^{-s}N_4^{1-s}.$$
For the above four sub-cases, we sum directly in $N_2,N_3, N_4$. Collecting all, we conclude that $\textup{Term}_1'=O(N^{-1+})$.

\subsection{$\textup{Term}_2'$}
As we did for $\textup{Term}_1'$, we write $\textup{Term}_2'$ as sum of 
\begin{equation}\label{DyadicPiece2}
\Big(1-\frac{m(N_1)}{m(N_2)m(N_3)m(N_4)}\Big)\int_0^\delta\int_{\mathbb{R}^3}\overline{w_1}v_2\overline{v_3}v_4 dxdt,
\end{equation}
where $w_1=P_{N_1}\mathcal{I}(|u|^2u)$ and $v_i=Q_{N_i}\mathcal{I}u$ for $i=2,3,4$. By symmetry, we may restrict to the case where $N_2\geq N_3\geq N_4$. \\
\textbf{($\textup{Term}_2'$, Case 1: $N_2\ll N$)} In this case, the symbol is zero, and thus $(\ref{DyadicPiece2})=0$.\\
\textbf{($\textup{Term}_2'$, Case 2: $N_2\gtrsim N$).} Applying $(5.12)$ with $m(N_1)\leq1$ and the Plancherel theorem, we write
$$|(\ref{DyadicPiece2})|\lesssim(m(N_2)m(N_3)m(N_4))^{-1} \| \overline{w_1}v_2\overline{v_3}v_4\|_{L_{t,x}^1}.$$
Consider the following three sub-cases.\\
\textbf{(Case 2a: $N_2\gtrsim N\gg N_3\geq N_4$)} By the H\"older inequality, the Sobolev inequality and Lemma 5.3, we get
\begin{align*}
|(\ref{DyadicPiece2})|&\leq m(N_2)^{-1} \|w_1\|_{L_t^{2-}L_x^2}\|v_2\|_{L_t^{8/3}L_x^4} \|v_3\|_{L_t^{16}L_x^8} \|v_4\|_{L_t^{16+}L_x^8}\\
&\lesssim\frac{N_2^{1-s}}{N^{1-s}} \delta^{0+}\|u\|_{L_{t,x}^6}^3N_2^{-1}\|\nabla\mathcal{I}u\|_{L_t^{8/3}L_x^4} \|\nabla\mathcal{I}u\|_{L_t^{16}L_x^{24/11}} N_4^{0+}\|\nabla\mathcal{I}u\|_{L_t^{16+}L_x^{24/11-}}\\
&\lesssim \delta^{0+}N^{-(1-s)}N_2^{-s}N_4^{0+}.
\end{align*}
\textbf{(Case 2b: $N_2\geq N_3\gtrsim N\gg N_4$).} Taking $\|w_1\|_{L_t^{2-}L_x^2}$, $\|v_2\|_{L_t^{8/3}L_x^4}$, $\|v_3\|_{L_t^{16}L_x^8}$ and $\|v_4\|_{L_t^{16+}L_x^8}$, we get
\begin{align*}
|(\ref{DyadicPiece2})|&\lesssim\frac{N_2^{1-s}N_3^{1-s}}{N^{2(1-s)}} \delta^{0+}N_2^{-1}N_4^{0+}=\delta^{0+}N^{-2(1-s)}N_2^{-s}N_3^{1-s}N_4^{0+}.
\end{align*}
\textbf{(Case 2c: $N_2\geq N_3\geq N_4\gtrsim N$).} Taking $\|w_1\|_{L_{t,x}^2}$, $\|v_2\|_{L_t^{8/3}L_x^4}$, $\|v_3\|_{L_t^{16}L_x^8}$ and $\|v_4\|_{L_t^{16}L_x^8}$, we get
$$|(\ref{DyadicPiece2})|\lesssim \frac{N_2^{1-s}N_3^{1-s}N_4^{1-s}}{N^{3(1-s)}}N_2^{-1}=N^{-3(1-s)}N_2^{-s}N_3^{1-s}N_4^{1-s}.$$
In each case, summing in $N_2, N_3, N_4$, we conclude that $\textup{Term}_2'=O(N^{-1+})$.

\appendix

\section{Zero Resonance}

\begin{lemma}[Absence of zero resonance]
If $V\in L^{3/2,\infty}$ and $\|V_-\|_{\mathcal{K}}<4\pi$, then zero is not a resonance.
\end{lemma}

\begin{proof}
Suppose that $\psi=-(-\Delta)^{-1}V\psi\neq0$ in $L^\infty$. We will deduce a contradiction by
$$\Big(1-\frac{\|V_-\|_{\mathcal{K}}}{4\pi}\Big)\|\nabla f\|_{L^2}\leq \la H f, f\ra_{L^2}$$
(see Lemma 2.1 of \cite{DFVV}). Pick a smooth cut-off $\chi$ such that  $\chi=1$ if $|x|\leq 1$ and $\chi=0$ if $|x|\geq 2$, and let $\chi_R:=\chi(\frac{\cdot}{R})$. Plugging $\psi\chi_R$ into the above inequality, we write
$$\|\nabla (\psi\chi_R)\|_{L^2}^2\lesssim\la H(\psi\chi_R), \psi\chi_R\ra_{L^2}=\la (H\psi)\chi_R-2\nabla\psi\cdot\nabla\chi_R-\psi\Delta\chi_R, \psi\chi_R\ra_{L^2}.$$
For $\epsilon>0$, choose $\psi_\epsilon\in C_c^\infty$ such that $\|\psi\chi_R^2-\psi_\epsilon\|_{L^{3,1}}<\epsilon$. Then, we have
\begin{align*}
\la H\psi, \psi_\epsilon\ra_{L^2}&=\la\psi, (-\Delta+V)\psi_\epsilon\ra_{L^2}=\la\psi, (I+V(-\Delta)^{-1})(-\Delta \psi_\epsilon)\ra_{L^2}\\
&=\la(I+(-\Delta)^{-1}V)\psi,(-\Delta \psi_\epsilon)\ra_{L^2}=0,\\
|\la H\psi, \psi\chi_R^2-\psi_\epsilon\ra_{L^2}|&\leq\epsilon\|H\psi\|_{L^{3/2,\infty}}\leq\epsilon\|\Delta\psi\|_{L^{3/2,\infty}}+\epsilon\|V\psi\|_{L^{3/2,\infty}}\leq2=\|V\psi\|_{L^{3/2,\infty}}\lesssim\epsilon.
\end{align*}
Since $\epsilon>0$ is arbitrary, this proves $\la H\psi, \psi\chi_R^2\ra_{L^2}=0$. Therefore,
\begin{align*}
\|\nabla (\psi\chi_R)\|_{L^2}^2&\lesssim\la -2\nabla\psi\cdot\nabla\chi_R-\psi\Delta\chi_R, \psi\chi_R\ra_{L^2}\\
&\leq \int_{\mathbb{R}^3}|\psi|^2 |\nabla\cdot((\nabla\chi_R)\chi_R)|dx+\int_{\mathbb{R}^3}|\psi|^2 |\Delta\chi_R||\chi_R| dx\lesssim R^{-1/2}\|\psi\|_{L^\infty}^2.
\end{align*}
Sending $R\to+\infty$, we conclude that $\psi\equiv0$ (contradiction!). 
\end{proof}

\end{document}